\documentclass[12pt, a4paper]{amsart}
\usepackage{graphicx} 
\usepackage{amsmath,amsthm,amssymb,latexsym,a4wide,tikz,multicol,tikz-cd, calc, bm}
\usepackage{tikz-qtree,tikz-qtree-compat}
\usepackage{mathtools,stmaryrd}
\usepackage[mathscr]{euscript}
\usepackage{float}

\usepackage[colorlinks]{hyperref}
\usepackage[margin=10pt,font=small,labelfont=bf, labelsep=period]{caption}
\usepackage{mathrsfs}
\usepackage{geometry} 
\usepackage[active]{srcltx}
\usepackage{cite}

\let\OLDthebibliography\thebibliography
\renewcommand\thebibliography[1]{
  \OLDthebibliography{#1}
  \setlength{\parskip}{0pt}
  \setlength{\itemsep}{0pt plus 0.2ex}
}

\usepackage{comment}
\usepackage{tikz}
\tikzset{font=\small}
\usepackage{enumerate}
\usetikzlibrary{matrix,arrows,automata,positioning,decorations.pathmorphing}
\usetikzlibrary{cd}
\usetikzlibrary{patterns}
\tikzset{main node/.style={circle,fill=white,draw,minimum size=0.1cm,inner sep=0pt},}

\newtheorem{theorem}{Theorem} [section]
\newtheorem{lemma}[theorem]{Lemma}

\newtheorem{proposition}[theorem]{Proposition}
\theoremstyle{definition}
\newtheorem{definition}[theorem]{Definition}

\newtheorem{example}[theorem]{Example}

\newcommand{\mx}[1]{#1^{\mathfrak{m}}}
\newcommand{\tr}{\operatorname{tr}}

\title[$F$-congruences on $F$-inverse monoids]{$F$-congruences on $F$-inverse monoids}

\usepackage[active]{srcltx}

\usepackage{enumerate}
\numberwithin{equation}{section}

\makeatletter
\def\l@subsection{\@tocline{2}{0pt}{2.5pc}{2.5pc}{}}
\makeatother

\begin{document}

\author{Ajda Lemut Furlani}
 \address{A. Lemut Furlani: Institute of Mathematics, Physics and Mechanics, Jadranska ulica 19, SI-1000 Ljubljana, Slovenia/ Faculty of Mathematics and Physics, Jadranska ulica 19, SI-1000 Ljubljana, Slovenia}
\email{ajda.lemut\symbol{64}imfm.si}
\thanks{The author was supported by the ARIS grant P1-0288.}

\sloppy

\begin{abstract} 
We study the class of congruences on $F$-inverse monoids that respect the unary operation $a\mapsto\mx{a}$, assigning to each element the maximum element of its $\sigma$-class, which we refer to as $F$-congruences.
We characterize kernel normal systems and congruence pairs on
$F$-inverse monoids which give rise to such congruences 
and establish a number of 
further related results including  a classification of $F$-congruence-free $F$-inverse monoids.

\vspace{0.1cm}


\vspace{0.1cm}

\end{abstract}

\maketitle

\section{Introduction}
The study of congruences on 
semigroups plays 
a crucial role in understanding their structure. Congruences on inverse semigroups have been characterized in terms of kernel normal systems by Preston \cite{Preston54} and congruence pairs by Petrich \cite{Petrich78}. Both of these approaches
extend the classical characterization of group congruences via normal subgroups. 

$F$-inverse monoids form an important class of inverse semigroups which arise naturally in various
mathematical contexts, see \cite{AKSz21} and the references therein. They form a variety in the enriched signature $(\cdot\,,{}^{-1},\mx{},1)$, as was observed by Kinyon in \cite{K18}, where the unary operation $\mx{}$ maps each element to the maximum element of its $\sigma$-class, with $\sigma$ denoting the minimum group congruence.
They have received much recent attention in the literature, see, e.g., \cite{AKSz21, KLF24, Sz23, DK26,KS26,ABO22}. 

The additional operation $\mx{}$ on $F$-inverse monoids yields a natural subclass of congruences that respect this operation, which we will refer to as $F$-congruences.
The aim of the present paper is to initiate the study of these congruences.

Specifically, we refine the characterizations of congruences on inverse semigroups by Preston and Petrich to the setting of $F$-congruences on $F$-inverse monoids
by determining the necessary and sufficient conditions for a kernel normal system or a congruence pair on an $F$-inverse monoid to give rise to an $F$-congruence (see Theorems~\ref{thm:fkernel} and \ref{thm:fpair}).
We also show that, apart from the universal congruence, a Rees congruence on an $F$-inverse monoid fails to be an $F$-congruence unless the monoid is a semilattice (see Proposition~\ref{prop:rees}).
We further describe the greatest $F$-congruence saturating a given subset and the greatest idempotent-separating $F$-congruence (see Propositions~\ref{prop:tildetau} and \ref{prop:tildemu}). Finally, motivated by earlier work on  congruence-free inverse semigroups (see, e.g., \cite{Trotter74,Baird75,Munn74,Reilly76}), we investigate 
$F$-congruence-free $F$-inverse monoids. We show that every such monoid is congruence-free and give their characterization (see Theorem~\ref{thm:nocongfree}).

For the undefined notions in the inverse semigroup theory we refer the reader to \cite{Lawsonbook, Petrichbook}.

\section{Preliminaries}\label{sec:preliminaries}

\subsection{Inverse semigroups and \texorpdfstring{$F$}{F}-inverse monoids}

Recall that an {\em inverse semigroup} is an algebra $(S; \cdot\,, ^{-1})$ such that $(S;\cdot)$ is a semigroup and the following identities hold:
\begin{equation*}
xx^{-1}x=x,\quad(x^{-1})^{-1}=x \quad \text{and} \quad xx^{-1}yy^{-1}=yy^{-1}xx^{-1}.
\end{equation*}
In what follows we omit the signature from the notation and write simply $S$ instead of $(S; \cdot\,, ^{-1})$.

Let $S$ be an inverse semigroup.
A {\em congruence} $\rho$ on $S$ is an equivalence relation that respects the multiplication, that is, for all $a,b,c\in S$, we have that $a\mathrel{\rho}b$ implies
\mbox{$ac\mathrel{\rho}bc$} and $ ca\mathrel{\rho}cb.$
The following lemma is well known. We provide it with a proof for completeness. 
\begin{lemma} \label{lem:inversion}
Let $\rho$ be a congruence on $S$ and let $a\mathrel{\rho}b$ for some $a,b\in S$. Then $a^{-1}\mathrel{\rho}b^{-1}$.  
\end{lemma}

\begin{proof}
We have \mbox{$a^{-1}=a^{-1}aa^{-1} \mathrel{\rho}a^{-1}ba^{-1}$.} Furthermore, 
$$
a^{-1}ba^{-1} = a^{-1}bb^{-1}ba^{-1} \mathrel{\rho} a^{-1}bb^{-1}aa^{-1}=a^{-1}aa^{-1}bb^{-1}=a^{-1}bb^{-1}\mathrel{\rho} a^{-1}ab^{-1}.  
$$
Similarly, $b^{-1}= b^{-1}bb^{-1}  \mathrel{\rho}b^{-1}ab^{-1}$ and 
$$
b^{-1}ab^{-1} =b^{-1}aa^{-1}ab^{-1}\mathrel{\rho}b^{-1}ba^{-1}ab^{-1} =a^{-1}ab^{-1}bb^{-1} =a^{-1}ab^{-1}.
$$
Hence $a^{-1}\mathrel{\rho}b^{-1}$.     
\end{proof}

The set $E(S)$ of idempotents of $S$ is a semilattice where $e\leq f$ if and only if $e = ef=fe$. The {\em natural partial order} on $S$ is defined by $s \leq t$ if and only if there exists $e \in E(S)$ such that $s=et$, or equivalently, there exists $f \in E(S)$ such that $s = tf$. By $\sigma$ we denote the {\em minimum group congruence on $S$}, that is, the minimum congruence on $S$ such that $S/\sigma$ is a group.
It is known that $s \mathrel{\sigma}t$ if and only if there exists $u \in S$, such that $u\leq s,t$, or, equivalently, if there exists some $e \in E(S)$ such that $es=et$.

An inverse semigroup $S$ is called {\em $E$-unitary} if, whenever $e\in E(S)$ and $e\leq s$, it follows that $s\in E(S)$. Equivalently, $S$ is $E$-unitary if $E(S)$ is a $\sigma$-class.
We say that $S$ is {\em $F$-inverse} if every $\sigma$-class has a maximum element with respect to the natural partial order.    
Every $F$-inverse semigroup is $E$-unitary and is a monoid, whose identity element is the maximum idempotent, see, e.g., \cite[Lemma~5.8.3]{Lawsonbook}. 
$F$-inverse monoids form a variety of algebras in the signature $(\cdot\,,^{-1},\mx{},1)$, where the unary operation $\mx{}$  maps an element $x$ to the maximum element $\mx{x}$ of its $\sigma$-class, see \cite{AKSz21}.
This variety is defined by the identities defining the variety of inverse monoids together with the 
identities:
\begin{equation*}
x = \mx{x}x^{-1}x\quad \text{and} \quad \mx{x} = \mx{(xy^{-1}y)}.
\end{equation*}

In the sequel, we will need the following lemma.

\begin{lemma}{\cite[Lemma~3.3]{AKSz21}}\label{lem:Fid2}
Every $F$-inverse monoid satisfies the following identities:
\begin{enumerate}
\item $(\mx{x})^{-1} = \mx{(x^{-1})}$;
\item $\mx{x}\mx{y} = \mx{(xy)}(\mx{y})^{-1}\mx{y} = \mx{x}(\mx{x})^{-1}\mx{(xy)}$;
\item $\mx{(x_0\mx{y_1}x_1\cdots x_{n-1}\mx{y_n}x_n)}= \mx{(x_0y_1x_1\cdots x_{n-1}y_nx_n)}$ for every $n\in\mathbb{N}$.
\end{enumerate}    
\end{lemma}
We also record the following known observation. 
\begin{lemma}\label{lem:13j}
Let $S$ be an $F$-inverse monoid and $a,b\in S$. Then
$\mx{(ab)}\geq\mx{a}\mx{b}$.
\end{lemma}

\begin{proof}
Since $\mx{(ab)}\mathrel{\sigma}ab\mathrel{\sigma}\mx{a}\mx{b}$ and  $\mx{(ab)}$ is the maximum element of the $\sigma$-class of $ab$, we have $\mx{(ab)}\geq\mx{a}\mx{b}$.
\end{proof}

\subsection{Kernel normal systems}\label{subsec:kns}
In this section we summarize the characterizations of congruences on inverse semigroups in terms of kernel normal systems due to Preston \cite{Preston54}. We use the terminology and notation of \cite{Petrichbook}.

Let $S$ be an inverse semigroup. 
A family $\mathscr{K}$ of pairwise disjoint inverse subsemigroups of $S$ is called a {\em kernel normal system} for $S$ if it satisfies the following conditions:
\begin{enumerate}
\item[(K1)] $E(S)\subseteq \bigcup_{K\in \mathscr{K}} K$;
\item[(K2)] for each $a\in S$ and $K\in \mathscr{K}$, there exists $L\in \mathscr{K}$, such that $a^{-1}Ka\subseteq L$;
\item[(K3)] if $a,ab,bb^{-1}\in K$ for some $K\in \mathscr{K}$ and $a,b\in S$, then $b\in K$.
\end{enumerate}
Note that if $S$ is a group, then any kernel normal system for $S$ consists of a single normal subgroup.
For a kernel normal system $\mathscr{K}$ for $S$, define the relation $\xi_\mathscr{K}$ on $S$ by 
$$
\forall a,b\in S: \quad a\mathrel{\xi_\mathscr{K}}b \quad\Leftrightarrow \quad  aa^{-1},\quad bb^{-1},\quad ab^{-1}\in K \text{ for some } K\in \mathscr{K}.
$$
Next, let $\rho$ be a congruence on $S$ and define 
$$\mathscr{K}(\rho)=\{[e]_\rho \colon e\in E(S)\}.$$

We now state Preston's characterization of 
congruences on an inverse semigroup.

\begin{theorem}{\cite[Theorem~1]{Preston54}}\label{thm:preston}
Let $S$ be an inverse semigroup. 
\begin{enumerate}
\item If $\mathscr{K}$ is a kernel normal system for $S$, then $\xi_\mathscr{K}$ is a congruence on $S$.
\item If $\rho$ is a congruence on $S$, then $\mathscr{K}(\rho)$ is a kernel normal system for $S$.
\item The map $\rho\mapsto \mathscr{K}(\rho)$ is a bijection between congruences and kernel normal systems for $S$ with the inverse $\mathscr{K}\mapsto \xi_\mathscr{K}$.
\end{enumerate}
\end{theorem}

We will also need the following observation from \cite{Preston54}.

\begin{lemma}\label{lem:KNS}
Let $\mathscr{K}$ be a kernel normal system for $S$. Then, for all $a,b \in \bigcup_{K \in \mathscr{K}} K$,
$$
a \mathrel{\xi_\mathscr{K}} b  \quad\Leftrightarrow \quad a,b \in K \text{ for some } K \in \mathscr{K}.$$
\end{lemma}

\begin{proof}
If $a,b \in K$ for some $K \in \mathscr{K}$, then, since $K$ is an inverse subsemigroup, we have 
$aa^{-1},\, bb^{-1},\, ab^{-1} \in K$ and thus $a \mathrel{\xi_\mathscr{K}} b$.
Conversely, suppose $a,b \in \bigcup_{K \in \mathscr{K}} K$ and $a \mathrel{\xi_\mathscr{K}} b$. Let $a \in K$ for some $K \in \mathscr{K}$. 
Then $aa^{-1}\in K$ and since $a \mathrel {\xi_\mathscr{K}} b $, also $bb^{-1},ab^{-1}\in K$. Applying the condition (K3), we obtain $b^{-1}\in K$ and thus also $b\in K$.   
\end{proof}

\subsection{Congruence pairs}\label{subsec:cp}
We now summarize the characterization of congruences on inverse semigroups in terms of congruence pairs introduced by Petrich \cite{Petrich78}.
As in the previous section, we follow \cite{Petrichbook}.

Let $S$ be an inverse semigroup and $\rho$ a congruence on $S$. 
The {\em trace of $\rho$}, denoted $\tr\rho$, is the restriction of $\rho$ to the set $E(S)$ and the {\em kernel of $\rho$} is defined by
$$
\ker\rho = \{a \in S \colon a \mathrel{\rho} e \text{ for some } e \in E(S)\}.
$$
If $S$ is a group, 
$\tr\rho$ is the trivial congruence on $E(S)=\{1\}$, and the kernel is a normal subgroup.
A subsemigroup $K$ of $S$ is called a {\em  normal subsemigroup} if $E(S) \subseteq K$ and $aKa^{-1} \subseteq K$ for all $a \in S$. A congruence $\tau$ on $E(S)$ is called a {\em normal congruence} if for all $e,f\in E(S)$ and $a \in S$, we have that $e \mathrel{\tau} f$ implies  $a^{-1}ea \mathrel{\tau} a^{-1}fa$. Note that $\tr\rho$ is a normal congruence on $E(S)$ and $\ker\rho$ is a normal subsemigroup of $S$.
Finally, a pair $(K,\tau)$ is said to be a {\em congruence pair} for $S$  if it satisfies the following conditions:
\begin{enumerate}
\item[(P1)] $K$ is a normal subsemigroup of $S$;
\item[(P2)] $\tau$ is a normal congruence on $E(S)$ and the following conditions hold:
    \begin{enumerate}[(i)]
        \item for all $a\in S$, $e\in E(S)$, if $ae \in K$ and $e \mathrel{\tau} a^{-1}a$, then $a \in K$;
        \item for all $a\in S$, if $a \in K$, then $aa^{-1} \mathrel{\tau} a^{-1}a$.
    \end{enumerate}
\end{enumerate}
For a congruence pair $(K,\tau)$ for $S$, define the relation $\rho_{(K,\tau)}$ on $S$ by
$$
\forall a,b\in S: \quad a \mathrel{\rho_{(K,\tau)}}b 
\quad \Leftrightarrow  \quad
a^{-1}a \mathrel{\tau} b^{-1}b 
\quad \text{and} \quad
ab^{-1} \in K.
$$

The following characterization of congruences on inverse semigroups is due to Petrich.

\begin{theorem}{\cite[Theorem 4.4]{Petrich78}}\label{thm:Petrich}
Let $S$ be an inverse semigroup.
\begin{enumerate}
\item If $(K,\tau)$ is a congruence pair for $S$, then $\rho_{(K,\tau)}$ is a congruence on $S$.
\item If $\rho$ is a congruence on $S$, then $(\ker\rho, \tr\rho)$ is a congruence pair for $S$.
\item The map $\rho\mapsto (\ker\rho, \tr\rho)$ is a bijection between congruences and congruence pairs for $S$ with the inverse $(K,\tau)\mapsto \rho_{(K,\tau)}$.
\end{enumerate}
\end{theorem}

\section{Characterizations of \texorpdfstring{$F$}{F}-congruences}\label{sec:Fcharacterizations} 

Recall from Lemma \ref{lem:inversion} that every congruence on an inverse semigroup respects the inverse operation. A congruence on an $F$-inverse monoid, however, does not need to respect the additional $\mx{}$-operation. It is therefore natural to study the class of congruences that do respect it. In this section, we describe such congruences using the characterizations of congruences on inverse semigroups by Preston and Petrich recalled in Sections \ref{subsec:kns} and \ref{subsec:cp}, respectively. 

\subsection{\texorpdfstring{$F$}{F}-congruences}\label{sec:fcong}
From now on we focus on congruences on  $F$-inverse monoids with respect to the enriched signature $(\cdot\,,^{-1},\mx{},1)$. 

\begin{definition}($F$-congruence)
A congruence $\rho$ on an $F$-inverse monoid $S$ will be called an {\em  $F$-congruence} if, for all $a,b\in S$, we have that $a\mathrel{\rho}b$ implies $\mx{a}\mathrel{\rho}\mx{b}$.
\end{definition}

We denote the trivial congruence on an inverse semigroup $S$ by ${\mathrm{id}}_S$ and the universal congruence on $S$ by $\Delta_S$.
Note that if $S$ is an $F$-inverse monoid, then ${\mathrm{id}}_S$ and $\Delta_S$ are clearly $F$-congruences. 
Moreover, we have the following.

\begin{lemma}\label{lemma:Fsigma}
Let $\rho$ be a congruence on an $F$-inverse monoid $S$ such that either $\rho\subseteq \sigma$ or $\sigma \subseteq \rho$. Then $\rho$ is an $F$-congruence. 
In particular, the minimum group congruence $\sigma$ is an $F$-congruence.
\end{lemma}

\begin{proof}
Suppose first that $\rho\subseteq \sigma$. If $a\mathrel{\rho}b$ for some $a,b\in S$, then $a\mathrel{\sigma} b$. Hence, $\mx{a}=\mx{b}$ and $\rho$ respects the $\mx{}$-operation.

Suppose now that $\sigma \subseteq \rho$ and let $a\mathrel{\rho} b$ for some $a,b\in S$. Since $a\mathrel{\sigma}\mx{a}$ and $b\mathrel{\sigma}\mx{b}$ 
it follows that $a\mathrel{\rho}\mx{a}$ and $b\mathrel{\rho}\mx{b}$. Hence,  $\mx{a}\mathrel{\rho}\mx{b}$.
\end{proof}

The following example shows that an $F$-congruence does not necessarily need to contain $\sigma$ or be contained in $\sigma$.

\begin{example}
Let $Y=\{0,1\}$ be the two-element semilattice 
with $0<1$ and let $G=\{1_G,g\}$ be the cyclic group of order $2$. Then
the direct product $Y\times G$ is an $F$-inverse monoid
and the congruence induced by $\Delta_G$ is an $F$-congruence with classes $\{(0,1_G),(0,g)\}$ and  $\{(1,1_G),(1,g)\}$. However, since the 
$\sigma$-classes are $\{(0,1_G),(1,1_G)\}$ and $\{(0,g),(1,g)\}$ it is neither contained in $\sigma$ nor it contains $\sigma$.
\end{example}

Recall that a {\em semilattice monoid} is a semilattice
which has a maximum element.
We now observe that for groups and semilattice monoids the notions of congruences and $F$-congruences coincide.

\begin{lemma}\label{lem:j23a}
Let $S$ be either a group or a semilattice monoid. Then $S$ is an $F$-inverse monoid and every congruence on $S$ is an $F$-congruence.
\end{lemma}

\begin{proof}
Suppose first that $S$ is a group. Then $\mx{a}=a$ for every $a\in S$, and hence $S$ is $F$-inverse. Since $\mx{}$ is the identity operation, every congruence on $S$ is an $F$-congruence.

Suppose now that $S$ is a semilattice monoid. Then $\mx{a}=1$ for every $a\in S$, and hence $S$ is $F$-inverse. Since $\mx{a}=1$ for all $a\in S$, it easily follows that every congruence on $S$ is also an $F$-congruence.
\end{proof}

\subsection{\texorpdfstring{$F$}{F}-kernel normal systems}\label{sec:kernelnormal}

Let $S$ be an $F$-inverse monoid and $\rho$ an $F$-congruence on $S$. Since $\rho$ is a congruence, 
we have $\rho = \xi_{\mathscr{K}}$ for some kernel normal system $\mathscr{K}$ for $S$ (see Section \ref{subsec:kns}). In this section we determine which kernel normal systems give rise to $F$-congruences. 

We first observe the following.
\begin{lemma}\label{lem:necessary1}
Let $\rho$ be an $F$-congruence on an $F$-inverse monoid $S$ and let $a\mathrel{\rho} e$ for some $a\in S$ and $e\in E(S)$. Then $\mx{a}\mathrel{\rho} 1$.
\end{lemma}

\begin{proof}
Since $a\mathrel{\rho} e$ and $\rho$ respects the $\mx{}$-operation, we have $\mx{a}\mathrel{\rho} \mx{e}=1$.
\end{proof}

This lemma motivates the following definition.
\begin{definition}($F$-kernel normal system)
A kernel normal system $\mathscr{K}$ for $S$ will be called an {\em $F$-kernel normal system} if the following condition holds:
\begin{enumerate}
\item[(FK)] if $ a\in K$ for some $K\in\mathscr{K} $, then $\mx{a}\in L$, where $L\in \mathscr{K}$ is such that $1\in L$.
\end{enumerate}
\end{definition}

\begin{theorem}\label{thm:fkernel} A congruence $\rho$ on an $F$-inverse monoid $S$ is an $F$-congruence if and only if $\mathscr{K}(\rho)$ is an  $F$-kernel normal system for $S$. Consequently, the map $\rho\mapsto \mathscr{K}(\rho)$ is a bijection between $F$-congruences and $F$-kernel normal systems for $S$ with the inverse $\mathscr{K}\mapsto \xi_\mathscr{K}$.
\end{theorem}

\begin{proof}
Let $\rho$ be a congruence on $S$ and suppose first that it is an $F$-congruence. Then, by \mbox{Theorem \ref{thm:preston}}, $\mathscr{K}(\rho)$ is a kernel normal system for $S$, so it remains to verify that $\mathscr{K}(\rho)$ satisfies the condition (FK). Let $K\in \mathscr{K}(\rho)$ and $a\in K$. By the definition of $\mathscr{K}(\rho)$ we have that $K= [e]_\rho$ for some $e\in E(S)$. It follows that $a\mathrel{\rho}e$, and hence, by Lemma~\ref{lem:necessary1}, $\mx{a}\mathrel{\rho}1$. Therefore, taking $L=[1]_\rho$, we have $1,\mx{a}\in L$ so that $\mathscr{K}(\rho)$ satisfies (FK).

Conversely, suppose $\rho$ is a congruence 
on $S$ such that $\mathscr{K}(\rho)$ is an  $F$-kernel normal system for $S$. We show that $\rho$ is an $F$-congruence. Let $a\mathrel{\rho}b$ for some $a,b\in S$. We show that $\mx{a}\mathrel{\rho}\mx{b}$. By Theorem~\ref{thm:preston} we have 
$\rho=\xi_{\mathscr{K}(\rho)}$. Since $a\mathrel{\rho}b$, we have
$$
aa^{-1},\quad bb^{-1},\quad ab^{-1}\in K, \quad\text{for some } K\in \mathscr{K(\rho)}.
$$
We aim to show that 
$$
\mx{a}(\mx{a})^{-1},\quad \mx{b}(\mx{b})^{-1},\quad \mx{a}(\mx{b})^{-1}\in K', \quad \text{for some } K'\in \mathscr{K}(\rho).
$$
Applying Lemma \ref{lem:Fid2}(1) and Lemma \ref{lem:Fid2}(2) to $\mx{a}(\mx{b})^{-1}$, 
we write
\begin{equation}\label{eq:j14}
\mx{a}(\mx{b})^{-1} =\mx{a}\mx{(b^{-1})}=
\mx{(ab^{-1})} \mx{b}(\mx{b})^{-1}.    
\end{equation}
Further, by (FK) there exists $L\in\mathscr{K}(\rho)$ such that $\mx{(ab^{-1})},1\in L$ so that by Lemma \ref{lem:KNS} $\mx{(ab^{-1})}\mathrel{\rho}1.$
Applying this to \eqref{eq:j14}, we obtain
$$
\mx{(ab^{-1})} \mx{b}(\mx{b})^{-1}\mathrel{\rho}1\mx{b}(\mx{b})^{-1}=\mx{b}(\mx{b})^{-1},$$
and thus
\begin{equation}\label{eq:j1}
\mx{a}(\mx{b})^{-1}
\mathrel{\rho} \mx{b}(\mx{b})^{-1}.
\end{equation}
By symmetry we also have
\begin{equation}\label{eq:j3}
\mx{b}(\mx{a})^{-1} 
\mathrel{\rho} \mx{a}(\mx{a})^{-1}.
\end{equation}
Next, since $\rho$ respects the inversion, 
it follows from \eqref{eq:j1} that
\begin{equation*}\label{eq:j2}
\mx{b}(\mx{a})^{-1}= (\mx{a}(\mx{b})^{-1})^{-1}\mathrel{\rho}(\mx{b}(\mx{b})^{-1})^{-1}=\mx{b}(\mx{b})^{-1}.
\end{equation*}
Together with \eqref{eq:j1} and \eqref{eq:j3}, this yields 
$$
\mx{a}(\mx{b})^{-1}\mathrel{\rho}\mx{b}(\mx{b})^{-1} 
\mathrel{\rho}\mx{a}(\mx{a})^{-1}.
$$
Applying Lemma \ref{lem:KNS}, we now conclude that the elements $\mx{a}(\mx{a})^{-1},\,\mx{b}(\mx{b})^{-1}$ and $\mx{a}(\mx{b})^{-1}$ all belong to the same $ K'\in \mathscr{K}(\rho)$. Hence, $\mx{a}\mathrel{\rho}\mx{b}$ by definition of $\xi_{\mathscr{K}(\rho)}=\rho$.
\end{proof}

\subsection{\texorpdfstring{$F$}{F}-congruence pairs}\label{sec:trker}
In this section we provide a characterization of $F$-congruences on $F$-inverse monoids based on the characterization of congruences by Petrich via congruence pairs (see Section \ref{subsec:cp}).
Motivated by the additional condition (FK) on kernel normal systems, we define the following.

\begin{definition}($F$-congruence pair)
A congruence pair $(K,\tau)$ for an $F$-inverse monoid $S$ will be called an {\em $F$-congruence pair} 
if the following conditions hold:
\begin{enumerate}\label{eq:mclosed_kernel}
\item[(FP1)] if $a\in K$, then  $\mx{a}\in K$;
\item[(FP2)] if $a\in K$, then $(\mx{a})^{-1}\mx{a}\mathrel{\tau} 1$.
\end{enumerate}
\end{definition}

In the proof of the characterization theorem, we will need the following property of $F$-congruence pairs.

\begin{lemma}\label{lem:am1}
Let $(K,\tau)$ be an $F$-congruence pair for an $F$-inverse monoid $S$. Then, for every $a \in K$, we have $\mx{a} \mathrel{\rho_{(K,\tau)}} 1$.
\end{lemma}

\begin{proof}
It suffices to verify that $\mx{a}\in K$ and $(\mx{a})^{-1}\mx{a}\mathrel{\tau}1$, but these are exactly the conditions (FP1) and (FP2).
\end{proof}

We are now ready to prove the main result of this section. 

\begin{theorem}\label{thm:fpair}
A congruence $\rho$ on an $F$-inverse monoid $S$ is an $F$-congruence if and only if $(\ker\rho, \tr\rho)$ is an  $F$-congruence pair for $S$. Consequently, the map $\rho\mapsto (\ker\rho, \tr\rho)$ is a bijection between $F$-congruences and $F$-congruence pairs for $S$  with the inverse $(\ker\rho, \tr\rho)\mapsto \rho_{(\ker\rho, \tr\rho)}$.
\end{theorem}

\begin{proof}
Suppose first that $\rho$ is an $F$-congruence on $S$ and let $a \in \ker \rho$. Then there exists $e \in E(S)$ such that $a \mathrel{\rho} e$. 
Since $\rho$ respects the $\mx{}$-operation, it follows that 
$\mx{a} \mathrel{\rho} \mx{e} = 1$ and hence $\mx{a} \in \ker \rho$. Therefore, the congruence pair $(\ker\rho,\tr\rho)$ satisfies (FP1).
It remains to verify (FP2). Since $\mx{a} \mathrel{\rho} 1$, we have
$
(\mx{a})^{-1} \mx{a} \mathrel{\rho} 1^{-1}1 = 1.
$
As both $(\mx{a})^{-1}\mx{a}$ and $1$ are idempotents, it follows from the definition of the trace that
$(\mx{a})^{-1}\mx{a} \mathrel{\tr \rho} 1$,
as required.
 
For the reverse implication, suppose $\rho$ is a congruence 
on $S$ such that $(\ker \rho,\tr \rho)$  is an  $F$-congruence pair for $S$. We show that $\rho$ is an $F$-congruence.
Suppose $a\mathrel{\rho}b$ for some $a,b \in S$.
By Theorem~\ref{thm:Petrich} we have 
$\rho=\rho_{(\ker \rho,\tr \rho)}$.   
Therefore, bearing in mind that $a\mathrel{\rho}b$, we have
$$
a^{-1}a \mathrel{\tr \rho} b^{-1}b \quad \text{and} \quad ab^{-1} \in \ker \rho.
$$
We aim to prove $\mx{a}\mathrel{\rho}\mx{b}$.
Since $\rho=\rho_{(\ker \rho,\tr \rho)}$, it suffices to show that
$$
(\mx{a})^{-1}\mx{a} \mathrel{\tr \rho} (\mx{b})^{-1}\mx{b} 
\quad \text{and} \quad 
\mx{a}(\mx{b})^{-1} \in \ker \rho.
$$
We first show 
that $\mx{a}(\mx{b})^{-1} \in \ker \rho$.
Since $ab^{-1} \in \ker \rho$, by Lemma \ref{lem:am1} we have
\begin{equation}\label{eq:8j}
\mx{(ab^{-1})} \mathrel{\rho} 1. 
\end{equation}
Moreover, by Lemma \ref{lem:13j} we have $\mx{(ab^{-1})} \geq \mx{a}\mx{(b^{-1})}$,
so that there exists $e \in E(S)$ such that
$$
\mx{a}\mx{(b^{-1})} = \mx{(ab^{-1})}e 
\mathrel{\rho} 1e=e.
$$
It follows that $\mx{a}\mx{(b^{-1})} \in \ker \rho$. 
It remains to verify that 
$(\mx{a})^{-1}\mx{a} \mathrel{\tr \rho} (\mx{b})^{-1}\mx{b}.$
Applying Lemma \ref{lem:Fid2} and \eqref{eq:8j},
we write
$$\mx{a}(\mx{b})^{-1} = \mx{(ab^{-1})} \mx{b}(\mx{b})^{-1} \mathrel{\rho}\mx{b}(\mx{b})^{-1}$$
and by symmetry also
$
\mx{b}(\mx{a})^{-1} 
\mathrel{\rho} \mx{a}(\mx{a})^{-1}. $
Then 
$$
\mx{b}(\mx{b})^{-1}\mathrel{\rho}\mx{a}(\mx{b})^{-1}= (\mx{b}(\mx{a})^{-1})^{-1} \mathrel{\rho} (\mx{a}(\mx{a})^{-1})^{-1} =\mx{a}(\mx{a})^{-1}.
$$
Since $\mx{b}(\mx{b})^{-1}$ and $\mx{a}(\mx{a})^{-1}$ are idempotents, it follows from the definition of the trace that  $(\mx{a})^{-1}\mx{a} \mathrel{\tr \rho} (\mx{b})^{-1}\mx{b}$, as required.
\end{proof}

\section{Some special \texorpdfstring{$F$}{F}-congruences}

In this section we study several $F$-congruences on $F$-inverse monoids related to some notable congruences on inverse semigroups.

\subsection{Rees \texorpdfstring{$F$}{F}-congruences}\label{sec:rees} 

Let $S$ be an inverse semigroup and $I\subseteq S$ an ideal, which we assume to be non-empty.
The {\em Rees congruence associated with $I$}, denoted by $\rho_I$, is defined by
$$
\forall a,b\in S:\quad a\mathrel{\rho_I}b \quad\Leftrightarrow \quad  a=b \quad \text{or}\quad a,b\in I.
$$ 
Let now $S$ be an $F$-inverse monoid.
A Rees congruence on $S$, 
which is an $F$-congruence, will be called a {\em Rees $F$-congruence}. In what follows, we characterize Rees congruences that are $F$-congruences.

If $S$ is a semilattice monoid, then, by Lemma~\ref{lem:j23a}, every Rees congruence on $S$ is a Rees $F$-congruence. Hence, we restrict our attention to the case where $S\neq E(S)$.

\begin{lemma}\label{lem:Fideal}
Let $I$ be an ideal of an $F$-inverse monoid $S$ where $S\neq E(S)$.
Then $\rho_I$ is a Rees $F$-congruence if and only if $\mx{I}\subseteq I$.
\end{lemma}

\begin{proof}
Suppose $\rho_I$ is a Rees $F$-congruence and let $a\in I$. Choose an arbitrary $b\in S\setminus E(S)$. Since $I$ is an ideal, we have $ab\in I$, and hence $a\mathrel{\rho_I}ab$. As $\rho_I$ respects the $\mx{}$-operation, it follows that
$\mx{a}\mathrel{\rho_I}\mx{(ab)}.$
Therefore, either $\mx{a}=\mx{(ab)}$ or $\mx{a},\mx{(ab)}\in I$. 
We claim that the first possibility cannot occur. Indeed, $\mx{a}=\mx{(ab)}$ holds if and only if
$
[a]_\sigma=[ab]_\sigma=[a]_\sigma[b]_\sigma,
$
which is equivalent to $[b]_\sigma=[1]_\sigma$.
Since $S$ is $F$-inverse, it is
$E$-unitary, and hence $b\in E(S)$. This contradicts the choice of $b$.
It follows that $\mx{a}\in I$. Since $a\in I$ was arbitrary, we conclude that $\mx{I}\subseteq I$.

Conversely, suppose that $I$ satisfies $\mx{I}\subseteq I$ and let $a\mathrel{\rho_I}b$. If $a=b$, then $\mx{a}=\mx{b}$ and hence $\mx{a}\mathrel{\rho_I}\mx{b}$. Otherwise, $a,b\in I$. Since $\mx{I}\subseteq I$, it follows that $\mx{a},\mx{b}\in I$. Thus, $\mx{a}\mathrel{\rho_I}\mx{b}$ and $\rho_I$ is a Rees $F$-congruence.
\end{proof}

\begin{proposition}\label{prop:rees}
Let $S$ be an $F$-inverse monoid with $S\neq E(S)$. Then the only Rees $F$-congruence on $S$ is the universal congruence $\Delta_S$.
\end{proposition}

\begin{proof}
Suppose $I$ is an ideal such that $\rho_I$ is a Rees $F$-congruence and let $a\in I$. Then $I$ contains the idempotent $aa^{-1}$. By Lemma \ref{lem:Fideal} we have $\mx{I}\subseteq I$. Hence,  $\mx{(aa^{-1})}=1 \in I$. It follows that $I=S$, hence $\rho_I=\Delta_S$.
\end{proof}

\begin{proposition}\label{prop:a11}
Let $S$ be an $F$-inverse monoid. 
\begin{enumerate}
\item If $S$ has a zero element, then $S$ is a semilattice monoid.
\item The $F$-congruence ${\mathrm{id}}_S$ is a Rees $F$-congruence if and only if $S$ has a zero element.
\end{enumerate}
\end{proposition}

\begin{proof}
(1) Suppose that $S$ has a zero element $0$. Since $S$ is $E$-unitary, $E(S)$ is a $\sigma$-class. Moreover, $0s=0t$ for all $s,t\in S$, so that all elements of $S$ are $\sigma$-related to $0$. Therefore, the $\sigma$-class $E(S)$ must be equal to $S$, and hence $S$ is a semilattice monoid. 

(2) Suppose ${\mathrm{id}}_S=\rho_I$ for some ideal $I$. Since $I$ is a $\rho_I$-class, it follows that $I=\{a\}$ for some $a\in S$. Then, for every $b\in S$, we have that $ab,ba\in I$, and hence $ab=ba=a$. Thus, $a$ is a zero element of $S$.

Conversely, suppose that $S$ has a zero element $0$.
The $F$-congruence ${\mathrm{id}}_S$ is then a Rees congruence associated with the ideal $I=\{0\}$. Hence, it is a Rees $F$-congruence. 
\end{proof} 

Observe that, when an $F$-inverse monoid $S$ has a zero element, and thus, in view of Proposition \ref{prop:a11}(1), satisfies $S=E(S)$, the ideal $I=\{0\}$ does not satisfy the condition $\mx{I}\subseteq I$, since $1 \in \mx{I}\setminus I$.
It follows that the assumption 
$S\neq E(S)$ cannot be dropped from Lemma~\ref{lem:Fideal}.

\subsection{The greatest \texorpdfstring{$F$}{F}-congruence saturating a given subset}\label{sec:tildetau}

Recall that a subset $H$ of an inverse semigroup $S$ is said to be {\em saturated} by a congruence $\rho$ on $S$ if it is a union of $\rho$-classes. 
It is known and easy to show that the greatest congruence on $S$ saturating $H$ coincides with the {\em syntactic congruence} of $H$ (see \cite[Section 4.1]{Lawsonbook}) which is given by
$$
\forall a,b\in S: \quad a \mathrel{\tau_H} b 
\quad \Leftrightarrow  \quad
( xay \in H \,\, \Leftrightarrow \,\, xby \in H \quad \text{for all } x,y \in S^1).
$$
In this section, we describe the greatest $F$-congruence saturating a given subset of an $F$-inverse monoid.

\begin{proposition}\label{prop:tildetau}
Let $S$ be an $F$-inverse monoid and let $H\subseteq S$. 
Define the relation $\tilde{\tau}_H$ on $S$ as follows. For all $a,b\in S$, put
$a\mathrel{\tilde{\tau}_H} b$
if and only if 
$a\mathrel{\tau}_H b$
and the following condition holds:
$$
u\mx{(xay)}v \in H \,\,\Leftrightarrow\,\, u\mx{(xby)}v \in H \quad \text{for all $x,y,u,v \in S$}.
$$
Then $\tilde{\tau}_H$ is the greatest $F$-congruence on $S$ saturating $H$.
\end{proposition}

\begin{proof}
We first show that $\tilde{\tau}_H$ is an $F$-congruence.
It is easy to see that $\tilde{\tau}_H$ is an equivalence relation, so it remains to verify that it respects the multiplication and the $\mx{}$-operation.
Suppose $a \mathrel{\tilde{\tau}_H} b$ for some $a,b\in S$ and let 
$c \in S$. Since $\tau_H$ is a congruence, we immediately get 
$ac \mathrel{\tau_H} bc$ and $ca \mathrel{\tau_H} cb$. To see $ac \mathrel{\tilde{\tau}_H} bc$, we need to verify that $$u\mx{(xacy)}v \in H \,\,\Leftrightarrow\,\, u\mx{(xbcy)}v \in H \quad \text{for all $x,y,u,v \in S$}.$$
But as $a \mathrel{\tilde{\tau}_H} b$  and $x,cy,u,v \in S$, this is immediate. Similarly,  $ca \mathrel{\tilde{\tau}_H} cb$.  
We next verify that
$\tilde{\tau}_H$ respects the $\mx{}$-operation. 
Taking $x=y=1$ in the definition of $\tilde{\tau}_H$ we obtain 
$$
u\mx{a}v \in H \,\,\Leftrightarrow\,\, u\mx{b}v \in H \quad\text{for all $u,v \in S$},
$$
so that $\mx{a}\mathrel{\tau_H} \mx{b}$.
Further, by Lemma \ref{lem:Fid2}(3) we have $\mx{(xay)}=\mx{(x\mx{a}y)}$ and similarly $\mx{(xby)}=\mx{(x\mx{b}y)}$. Applying this to the definition of $\tilde{\tau}_H$ we obtain
$$
u\mx{(x\mx{a}y)}v \in H \,\,\Leftrightarrow\,\, u\mx{(x\mx{b}y)}v \in H \quad \text{for all } x,y,u,v \in S.
$$
Thus, $\mx{a} \mathrel{\tilde{\tau}_H} \mx{b}$ so that $\tilde{\tau}_H$ is an $F$-congruence.

It remains to show $\tilde{\tau}_H$ saturates $H$ and is maximal with this property.
Clearly, \mbox{$\tilde{\tau}_H \subseteq \tau_H$}. Since $\tau_H$ saturates $H$, it follows that $\tilde{\tau}_H$ also saturates $H$. Let now $\rho$ be another $F$-congruence on $S$ saturating $H$ and suppose $a\mathrel{\rho} b$ for some $a,b\in S$.
Since $\rho$ respects the multiplication,
we have \mbox{$xay\mathrel{\rho}xby$} for all $x,y\in S$. Further, since $\rho$ respects the $\mx{}$-operation, we have \mbox{$\mx{(xay)}\mathrel{\rho}\mx{(xby)}$} and finally also \mbox{$u\mx{(xay)}v\mathrel{\rho}u\mx{(xby)}v$} for all $x,y,u,v\in S$. Since $\rho$ saturates $H$, it follows that 
$$
xay \in H \,\,\Leftrightarrow\,\, xby \in H
\quad \text{and} \quad
u\mx{(xay)}v \in H \,\, \Leftrightarrow \,\, u\mx{(xby)}v \in H \quad \text{for all } x,y,u,v \in S.
$$
Thus, $\rho \subseteq \tilde{\tau}_H$, as needed.
\end{proof}

We now give an example which demonstrates that  $\tilde{\tau}_H$ can be strictly smaller than $\tau_H$.

\begin{example}\label{exm:j24}
Let $\mathcal{Y}=\{0,P,Q,1\}$
be the semilattice with the Hasse diagram shown in Figure~\ref{fig:Y} and let
$G=\{1_G,g\}$ be the cyclic group of order 2.
\begin{figure}[H]
\centering
\begin{tikzpicture}[scale=1]
    \node (1) at (0,2) {$1$};
    \node (P) at (-1,1.2) {$P$};
    \node (Q) at (1,1.2) {$Q$};
    \node (0) at (0,0.4) {$0$};

    \draw (0) -- (P);
    \draw (0) -- (Q);
    \draw (P) -- (1);
    \draw (Q) -- (1);
\end{tikzpicture}
\caption{The Hasse diagram of the semilattice $\mathcal{Y}$.}
\label{fig:Y}
\end{figure}
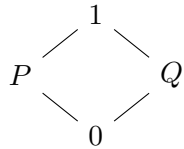
We consider the action of $G$ on $\mathcal{Y}$ by
order automorphisms as follows:
$$
g\cdot0=0,\quad g\cdot1=1, \quad
g\cdot P=Q,\quad g\cdot Q=P.
$$
Then the semidirect product $\mathcal{Y}\rtimes G$ is an $F$-inverse monoid with the operations given by
\begin{align*}
\forall i,j\in \{0,1\}\!\colon \quad(A,g^i)(B,g^j)&= (A\wedge (g^i\cdot B), g^{i+j}),\\
(A,g^i)^{-1}&= (g^{-i}\cdot A, g^{-i})=(g^i\cdot A,g^i),\\
\mx{(A,g^i)}&=(1,g^i)
\end{align*}
and the identity element being $(1,1_G)$. 

Let further $H=\{(1,g)\}.$
We first determine the $\tau_H$-classes. 
Let
$$a=(A,g^k),\quad b=(B,g^l)\in \mathcal{Y}\rtimes G.$$
By definition, $a\mathrel{\tau_H} b$ if, for any 
$x=(X,g^i),\, y=(Y,g^j)\in \mathcal{Y}\rtimes G$, we have that $xay\in H$ if and only if $ xby\in H$. We calculate:
$$
xay=(X\wedge (g^i\cdot A)\wedge (g^{i+k}\cdot Y), g^{i+k+j}),
$$
and similarly
$$
xby=(X\wedge (g^i\cdot B)\wedge (g^{i+l}\cdot Y), g^{i+l+j}).
$$
Since $\tau_H$ saturates $H$, we have that
$H=\{(1,g)\}$ is a $\tau_H$-class. We show that the remaining elements of $\mathcal{Y}\rtimes G$ split into two $\tau_H$-classes:
$$
H' = \{(1,1_G)\}\quad \text{and}\quad
H'' = \{(0,1_G),(P,1_G),(Q,1_G),(0,g),(P,g),(Q,g)\}.
$$

Suppose that $a=(A,g^k), b=(B,g^l)\in H''$.
Then $A,B \in\{0,P,Q\}$, which implies $g^i\cdot A,\, g^i \cdot B \in\{0,P,Q\}$, so that the first coordinates of $xay$ and $xby$ differ from $1$. Hence, neither of these elements belongs to $H$, and thus $a \mathrel{\tau_H}b$.

Suppose now that $a=(1,1_G)\in H'$ and $b=(B,g^l)\in H''$.
Then 
$$(1,1_G)a(1,g)=(1,g)\in H \quad \text{but} \quad (1,1_G)b(1,g)=(B,g^{l+1})\notin H.$$
Therefore, $a \not \mathrel{\tau_H}b$.
It follows that the $\tau_H$-classes are precisely $H$, $H'$, and $H''$.

To show that $\tilde{\tau}_H$ is properly contained in $\tau_H$, it suffices to show that $a\not\mathrel{\tilde{\tau}_H} b$ for some $a,b\in H''$. 
Let $a= (P,g), b=(P,1_G)\in H''$ and take $x=y=u=v=(1,1_G)$.
Then we have that
$$u\mx{(xay)}v = (1,g) \in H \quad \text{but} \quad u\mx{(xby)}v = (1,1_G) \notin H.$$
Hence, $a\not\mathrel{\tilde{\tau}_H}b$ and therefore $\tilde{\tau}_H\subsetneq\tau_H$. In fact, one can show that $H''$ splits  into two $\tilde{\tau}_H$-classes, namely
$\{(0,1_G),(P,1_G),(Q,1_G)\}$ and $\{(0,g),(P,g),(Q,g)\},$
which implies that the $F$-congruence $\tilde{\tau}_H$ on $ \mathcal{Y}\rtimes G$ is nontrivial.
\end{example}

\subsection{The greatest idempotent-separating \texorpdfstring{$F$}{F}-congruence}\label{sec:tildemu}

Recall that a congruence on a semigroup is called {\em idempotent-separating} if each its class contains at most one idempotent. The {\em greatest idempotent-separating congruence} $\mu$ on an inverse semigroup $S$ is 
characterized by: 
$$
\forall a,b\in S:\quad  a\mathrel{\mu} b
\quad\Leftrightarrow \quad 
aea^{-1}=beb^{-1}
\quad\text{for all } e\in E(S).
$$
In what follows, we describe the greatest idempotent-separating $F$-congruence on an $F$-inverse monoid.

\begin{proposition} \label{prop:tildemu}
Let $S$ be an $F$-inverse monoid.
Define the relation  $\tilde{\mu}$ on $S$ as follows. For all $a,b\in S$, put
$a\mathrel{\tilde{\mu}} b$ if and only if $a \mathrel{\mu}b$ and the following condition holds:
$$
\mx{(xay)}e(\mx{(xay)})^{-1}=\mx{(xby)}e(\mx{(xby)})^{-1}
\quad \text{for all } x,y\in S, e\in E(S).
$$
Then $\tilde{\mu}$ is the greatest idempotent-separating $F$-congruence.
\end{proposition}

\begin{proof}
We first show that $\tilde{\mu}$ is an $F$-congruence.
It is easy to see that $\tilde{\mu}$ is an equivalence relation, so we only verify it respects the multiplication and the $\mx{}$-operation. Suppose $a\mathrel{\tilde{\mu}} b$ for some $a,b\in S$ and let $c\in S$. 
We first prove that $ac \mathrel{\tilde{\mu}} bc.$
Since $a\mathrel{\mu} b$ and  $\mu$ is a congruence, we have $ac\mathrel{\mu} bc$.
Further, replacing $y$ by $cy$ in the definition of $\tilde{\mu}$ gives
$$
(xacy)^m e((xacy)^m)^{-1}
=
(xbcy)^m e((xbcy)^m)^{-1} \quad \text{for all }x,y\in S,\, e\in E(S).
$$
Hence, $ac \mathrel{\tilde{\mu}} bc$ and by a symmetric argument also  $ca \mathrel{\tilde{\mu}} cb$.
Finally, we show that
$
\mx{a}\mathrel{\tilde{\mu}}  \mx{b}.
$
Taking $x=y=1$ in the definition of $\tilde{\mu}$ immediately gives
$$
\mx{a} e(\mx{a})^{-1}
=
\mx{b} e(\mx{b})^{-1} \quad \text{for all } e\in E(S).
$$ 
Thus, $\mx{a} \mathrel{\mu}\mx{b}$.
Moreover, by Lemma \ref{lem:Fid2}(3) we have
$\mx{(xay)}=\mx{(x\mx{a} y)}$ and similarly  $\mx{(xby)}=\mx{(x\mx{b} y)}$. Therefore, $a \mathrel{\tilde{\mu}} b$ directly gives that
$$
\mx{(x\mx{a}y)}e(\mx{(x\mx{a}y)})^{-1}=\mx{(x\mx{b}y)}e(\mx{(x\mx{b}y)})^{-1} \quad \text{for all } x,y\in S,\, e\in E(S),
$$
and hence $\mx{a}\mathrel{\tilde{\mu}}  \mx{b}$, as needed.

It remains to show that $\tilde{\mu}$ is the greatest idempotent-separating $F$-congruence.
Since $\tilde{\mu}\subseteq \mu$ and $\mu$ is idempotent-separating, we have that $\tilde{\mu}$ is idempotent-separating as well.
Let now $\rho$ be an idempotent-separating $F$-congruence on $S$ and $a\mathrel{\rho}b$ for some $a,b\in S$. Since $\rho$ is a congruence, we have that $aea^{-1}\mathrel{\rho} beb^{-1}$ for all $e\in E(S)$. Moreover, $\rho$ also respects the $\mx{}$-operation, so that $\mx{(xay)}e(\mx{(xay)})^{-1}\mathrel{\rho}\mx{(xby)}e(\mx{(xby)})^{-1}$
for all  $x,y\in S$, $e\in E(S)$.
As $\rho$ is idempotent-separating, it follows that
$$
aea^{-1}= beb^{-1} \text{ and  } \mx{(xay)}e(\mx{(xay)})^{-1}=\mx{(xby)}e(\mx{(xby)})^{-1}\quad \text{for all } x,y\in S, \, e\in E(S).$$ Hence, $a\mathrel{\tilde{\mu}} b$ and therefore $\rho\subseteq \tilde{\mu}$.
\end{proof}

Recall that an inverse semigroup $S$ is called {\em fundamental} if $\mu= {\mathrm{id}}_S$.  
Since a congruence is idempotent-separating if and only if it is contained in the Green's relation $\mathcal{H}$, an inverse semigroup $S$ is fundamental if and only if ${\mathrm{id}}_S$ is the only congruence contained in $\mathcal{H}$.

We call an $F$-inverse monoid $S$ {\em $F$-fundamental} if 
$\tilde{\mu}= {\mathrm{id}}_S$, or, equivalently, if ${\mathrm{id}}_S$ is the only $F$-congruence contained in the Green's relation $\mathcal{H}$.

We now show that for the $F$-inverse monoid from Example~\ref{exm:j24} $\tilde{\mu}$ is the trivial congruence which is properly contained in $\mu$. Consequently, this provides an example of an $F$-inverse monoid which is $F$-fundamental but not fundamental.

\begin{example}\label{exm:j21}
Let $\mathcal{Y}\rtimes G$ be the $F$-inverse monoid from Example~\ref{exm:j24}.
It is easy to see that, for $(A,g^i),(B,g^j)\in \mathcal{Y}\rtimes G$, the Green's relations $\mathcal R$ and $\mathcal L$ are given by
\begin{align*}
(A,g^i) \mathrel{\mathcal{R}}(B,g^j) \quad&\Leftrightarrow \quad A=B, \\
(A,g^i) \mathrel{\mathcal{L}}(B,g^j)\quad&\Leftrightarrow \quad g^i \cdot A = g^j \cdot B.
\end{align*}
This implies that the $\mathcal H$-classes on $ \mathcal{Y}\rtimes G$ are:
$$
H_1 = \{(1,1_G),(1,g)\},\quad
$$
$$
H_2 = \{(P,1_G)\},\quad H_3=\{(P,g)\},\quad
H_4 = \{(Q,1_G)\},\quad H_5=\{(Q,g)\},
$$
$$
H_6 = \{(0,1_G),(0,g)\}.
$$
We now describe the congruences $\mu$ and $\tilde{\mu}$ on $ \mathcal{Y}\rtimes G$.
Since $\tilde{\mu}\subseteq\mu\subseteq\mathcal H$, the $\mathcal H$-classes
$H_2,\,H_3,\,H_4,$ and $H_5$ remain singleton classes for both $\mu$ and
$\tilde{\mu}$. Thus, it remains only to determine whether the elements of
$H_1$ and $H_6$ are $\mu$-related or $\tilde{\mu}$-related.
We first consider $(1,1_G),(1,g)\in H_1$. For $(P,1_G)\in E( \mathcal{Y}\rtimes G)$ we calculate:
$$
(1,1_G)(P,1_G)(1,1_G)^{-1}= (P,1_G),
$$
which differs from 
$$
(1,g)(P,1_G)(1,g)^{-1}= (g\cdot P,1_G)= (Q,1_G).
$$
Hence, $(1,1_G)\not\mu(1,g)$. 
Further, for any
$(C,1_G)\in E( \mathcal{Y}\rtimes G)$ we have:
$$
(0,1_G)(C,1_G)(0,1_G)^{-1}=(0,1_G)\quad \text{and} \quad (0,g)(C,1_G)(0,g)^{-1}=(0,1_G),
$$
so that $(0,1_G)\mathrel{\mu}(0,g)$. 
This shows that $\mu$ is a nontrivial congruence, with $H_6$ being the only nontrivial $\mu$-class.

We now show that the congruence $\tilde{\mu}$ on $ \mathcal{Y}\rtimes G$
is trivial.
For this, it suffices to show that $(0,1_G)$ and $(0,g)$ are not $\tilde{\mu}$-related.   
Taking $x=y=(1,1_G)$ and $e=(P,1_G)$ in the definition of $\tilde{\mu}$ yields
$$
\mx{(0,1_G)}(P,1_G)(\mx{(0,1_G)})^{-1} = (1,1_G)(P,1_G)(1,1_G)= (P,1_G),
$$
which differs from 
$$
\mx{(0,g)}(P,1_G)(\mx{(0,g)})^{-1} = (1,g)(P,1_G)(1,g)= (Q,1_G).
$$
Therefore, $(0,1_G)\not\mathrel{\tilde{\mu}} (0,g)$. 
Hence, $\tilde{\mu}$ is the trivial congruence. 
Together with the fact that $\mu$ is nontrivial, this implies that $\tilde{\mu}$ is properly contained in $\mu$.
\end{example}

\section{\texorpdfstring{$F$}{F}-congruence-free \texorpdfstring{$F$}{F}-inverse monoids}\label{sec:congruencefree}

Recall that an inverse semigroup $S$ is called {\em congruence-free} if it does not have any congruences apart from ${\mathrm{id}}_S$ and $\Delta_S$. 
Earlier work on such semigroups  \cite{Trotter74,Baird75,Munn74,Reilly76} and our study of $F$-congruences naturally 
leads to the question of which $F$-inverse monoids admit no $F$-congruences apart from the trivial and universal ones.  In this section, we provide a  characterization of such $F$-inverse monoids.

\begin{definition}($F$-congruence-free $F$-inverse monoid)
We call an $F$-inverse monoid $S$ {\em $F$-congruence-free} if it does not have any $F$-congruences apart from ${\mathrm{id}}_S$ and $\Delta_S$.
\end{definition}

Every $F$-congruence on an $F$-inverse monoid is by definition a congruence.
Hence, every congruence-free $F$-inverse monoid is $F$-congruence-free.

\begin{lemma}\label{lem:a10}
 Suppose $S$ is an $E$-unitary inverse semigroup.
\begin{enumerate}
\item If $\sigma= {\mathrm id}_S$, then $S$ is a group.
\item If $\sigma= \Delta_S$, then $S$ is a semilattice.
\end{enumerate}
\end{lemma}

\begin{proof}
If $\sigma={\mathrm{id}}_S$, then $S$ is clearly a group.
If $\sigma=\Delta_S$, then $S$ is a $\sigma$-class. Since $S$ is $E$-unitary, $E(S)$ is a $\sigma$-class, and hence $S=E(S)$ is a semilattice.
\end{proof}

\begin{theorem}\label{thm:nocongfree}
Let $S$ be an $E$-unitary inverse semigroup. The following statements are equivalent.
\begin{enumerate}
\item $S$ is an $F$-congruence-free $F$-inverse monoid.
\item $S$ is congruence-free.
\item $S$ is either a simple group or a semilattice with at most two elements.
\end{enumerate}  
\end{theorem}

\begin{proof}
$(1)\Rightarrow(2)$ 
Suppose $S$ is an $F$-congruence-free $F$-inverse monoid and suppose $\rho$ is a congruence on $S$ distinct from ${\mathrm{id}}_S$ and $\Delta_S$. Since the minimum group congruence $\sigma$ is an $F$-congruence, it must coincide with either ${\mathrm{id}}_S$ or $\Delta_S$. By Lemma~\ref{lem:a10}, it follows that $S$ is then
either a group or a semilattice. 
Then, by Lemma~\ref{lem:j23a}, $\rho$ is an $F$-congruence, contradicting the assumption that $S$ is $F$-congruence-free. Therefore, $S$ is congruence-free.

$(2)\Rightarrow(3)$ 
Suppose $S$ is congruence-free. 
Then the congruence $\sigma$ coincides with  one of the congruences ${\mathrm{id}}_S$ or $\Delta_S$, and hence, by Lemma~\ref{lem:a10}, $S$ is either a group or a semilattice. If $S$ is a group, the assumption implies that it must be simple. Suppose now that $S$ is a semilattice. 
If $S$ had three or more elements, there would exist $e\in S$ that is neither a least nor a greatest element. The ideal $\{f\in S\colon f\leq e\}$ would then be proper and nonzero, which would lead to a nontrivial and nonuniversal Rees congruence. Hence, $S$ must be a semilattice with at most two elements.

$(3)\Rightarrow(1)$
Let $S$ be a simple group or a semilattice with at most two elements. First observe that, by Lemma~\ref{lem:j23a}, $S$ is $F$-inverse.
Further, if $S$ is a simple group, it is clearly congruence-free. 
If $S$ is a semilattice with at most two elements, 
it is easy to see that its only congruences are the trivial and the universal ones.
Hence, in both cases, $S$ is congruence-free, 
and therefore also $F$-congruence-free.
\end{proof}

\section*{AI statement}
The results in this paper were obtained without AI assistance. The paper was written entirely by the author, with AI assistance used solely for minor language refinements.

\section*{Acknowledgments}
I am grateful to my PhD supervisor Ganna Kudryavtseva for suggesting this topic and for her valuable advice and guidance during the preparation of this paper.
I would also like to thank the referee for the useful comments and suggestions.

\end{document}